\documentclass[11pt]{amsart}
\usepackage{amssymb, amsfonts, amsmath, amsthm, rotating, dcolumn, ctable, tabularx}

\theoremstyle{plain}
\newtheorem{teo}{Theorem}[section]
\newtheorem{prop}[teo]{Proposition}
\newtheorem{cor}[teo]{Corollary}
\newtheorem{lem}[teo]{Lemma}

\theoremstyle{definition}
\newtheorem{defi}[teo]{Definition}

\theoremstyle{remark}
\newtheorem{rem}[teo]{Remark}

\input xy
\usepackage[all]{xy}

\newcommand{\z}  {\mathbf{Z}}
\newcommand{\ci} {\mathbf{C}}

\newcommand{\qu} {\mathbf{Q}}

\newcommand{\letr}[1] {\mathcal{1}}

\newcommand{\gq} {G_{\mathbf{Q}}}
\newcommand{\f } {\mathbf{F}}

\newcommand{\fbar} {\overline{\f}_p}

\DeclareMathOperator{\gl}{GL}
\DeclareMathOperator{\pgl}{PGL}
\DeclareMathOperator{\sll}{SL}
\DeclareMathOperator{\enn}{End}
\DeclareMathOperator{\M}{M}

\DeclareMathOperator{\fr}{Frob}
\DeclareMathOperator{\tr}{tr}

\DeclareMathOperator{\spe}{Spec}
\DeclareMathOperator{\Hom}{Hom}

\def\p{\mathfrak{p}}

\begin{document}

\title{Computing the number of certain Galois representations mod $p$}
\author{Tommaso Giorgio Centeleghe}

\address{Univesit\"at Duisburg--Essen, Fakult\"at f\"ur Mathematik.}
\email{{\tt tommaso.centeleghe@uni-due.de}}

\maketitle

\begin{abstract} Using the link between mod $p$ Galois representations of $\qu$
and mod $p$ modular forms established by Serre's Conjecture, we compute, for every prime
$p\leq 1999$, a lower bound for the number of isomorphism classes of continuous Galois representation
of $\qu$ on a two--dimensional vector space over $\fbar$ which are irreducible, odd, and unramified
outside $p$. 
\end{abstract}

\section{Introduction}

Let $p$ be a prime number and $\fbar$ an algebraic closure of $\f_p$, the finite
field with $p$ elements. Let $\gq$ denote the absolute Galois group of
$\qu$, with respect to the choice of an algebraic closure $\overline\qu$ of
$\qu$. An important consequence of (the level one case of) Serre's Modularity
Conjecture is the following finiteness theorem

\begin{teo}\label{fnt} There are only finitely many isomorphism classes of continuous
representations $\rho:\gq\rightarrow\gl_2(\fbar)$ that are irreducible, odd,
and unramified outside $p$.
\end{teo}

Continuity in this context means that $\rho$ has open kernel, compactness of $\gq$
implies that $\rho$ has finite image, and there exists a finite extension $\f(\rho)$ of $\f_p$ for
which a model of $\rho$ over $\f(\rho)$ can be found.
The statement obtained from the Theorem replacing $\fbar$ by a finite subfield $\f$
was known to be true classically as a consequence of the Hermite--Minkowski Theorem.
The point of Theorem \ref{fnt} is that for every prime $p$ one can find a finite subfield $\f$ of
$\fbar$ so that {\it all} the representations considered can be realized over $\f$.

Let $R(p)$ denote  the non--negative integer defined by Theorem \ref{fnt}. From the refined version
of Serre's Conjecture one immediately sees that $R(p)$ is bounded from above by a function $U(p)$ whose
behaviour with $p$ is $p^3/48+O(p^2)$ (cf. section \ref{computations}). Professor Khare
had raised the question of whether this upper bound give the correct asymptotic of $R(p)$ (cf. \cite{Kh1}, \S $8$), in our University of Utah
thesis we conjectured a positive answer to his question.

The conjecture predicts that congruences modulo $p$ between characteristic zero eigenforms of weight $k\leq p+1$
are ``rare" and that, moreover, the mod $p$ Galois representations of $\qu$ associated to classical cusp forms of level one tend to be irreducible and
wildly ramified at $p$.

In the computations presented here we collected a lower bound $L(p)$ of $R(p)$ for all primes $p\leq 1999$. The table at the
end of the paper contains the values of $L(p)$ that we had found, together with the upper bound $U(p)$ and with the ratio $(U(p)-L(p))/p^2$.
In the range explored the table shows a tendency for $R(p)$ to remain close to the upper bound $U(p)$.

Using the link between Galois representations and modular forms established by Serre's Conjecture, we computed $L(p)$ by estimating
the number of systems of Hecke eigenvalues arising from modular forms mod $p$ of level one. The method adopted is based on the
analysis of a {\it single} Hecke operator $T_n$ to deduce information about the mod $p$ arithmetic of the whole Hecke ring $\mathbf{T}^0_k$.
The software used to perform the computations is MAGMA.

%Craig Citro and Alex Ghitza have worked on a computational project very close
%to ours, the method used in their computations is different.

In section \ref{gener} standard results from the theory of modular forms and Galois representations are recalled,
the method used in the computations is explained in details in section \ref{computations}, and the remaining two sections provide the commutative
algebra needed for the computations.

The work presented in this paper started within my thesis project, I would like to express my gratitude to professor Khare for suggesting
this direction of research as well as for the invaluable attention that I have received from him. This paper benefitted from many interesting
conversation and advices that I received from professors Gebhard B\"ockle and Gabor Wiese during the past year. I am grateful to
them for their important help. I would like to thank professor Ulrich G\"ortz for letting me use the computer Pluto at the Institute for Experimental
Mathematics in Essen which performed the computations. I want to thank Craig Citro, who explained to me a lot on computing with modular forms.
Finally, the help of Panagiotis Tsaknias with the implementation of the algorithm and the production of the table was vital for me. I heartily thank him
for his kindness and availability.

\section{Generalities}\label{gener}

%We recall here several facts on the theory of modular forms (both classical and mod $p$)
%and on the Hecke operators associated to them.
In this preliminary section we adopt a very utilitarian point of view and recall all the results that we need
from the theory of modular forms (both classical and mod $p$) and their associated mod $p$ Galois representations.
For more details on modular forms on $\sll_2(\z)$ and their Hecke operators the reader can consult
\cite{La}. For an exposition of classical theorems linking mod $p$ modular forms to Galois representations, as well
as some more recent important development, the papers \cite{Ed} and \cite{Gr} are standard references.
We prefer not to say anything about Serre's Conjecture here. Instead, we will constantly keep this important
theorem in the back of our mind as the motivation for studying systems of Hecke eigenvalues arising from
modular forms mod $p$.
%
%conclude the section recalling the statement of the level one case of Serre's Conjecture. This important theorem,
%proved by Khare in 2005 (cf. \cite{Kh}),
%motivates the study of systems of Hecke eigenvalues arising from modular forms mod $p$.

Let $\M_k$ denote the space of classical modular forms of weight $k$ on the group
$\sll_2(\z)$, and let $\M_k^0$ be its cuspidal subspace. Denote by $\M_k(\z)$ (resp. $\M^0_k(\z)$)
the submodule of $\M_k$ (resp. $\M^0_k$) given by forms $f$ whose expansion at infinity
has integer coefficients. It is a basic fact that these submodules define integral structures, meaning that
the natural inclusions $\M_k(\z)\subset\M_k$ and $\M^0_k(\z)\subset\M^0_k$ induce isomorphisms
$\M_k(\z)\otimes\ci\simeq\M_k$ and $\M^0_k(\z)\otimes\ci\simeq\M^0_k$.

Let $p$ be a prime number. Following \cite{Se6}, we define the space $\M_k(\f_p)$ of modular forms mod $p$ of
weight $k$ on $\sll_2(\z)$ to be $\M_k(\z)/p\M_k(\z)$, similarly the cuspidal subspace is $\M^0_k(\f_p)=\M^0_k(\z)/p\M^0_k(\z)$.
If $p>3$, then these definitions agree with the geometric definitions \`a la Katz (\cite{Ka}, Theorem $1.8.2$).

For an integer $n>0$, the $n$--th Hecke operator on the space $\M_k^0$ is denoted by $T_n$, without
reference to the weight $k$. The Hecke operators all commute with each other, and
if $\ell_1,\ldots,\ell_r$ are the primes dividing $n$, the operator $T_n$ can be written
as a polynomial in the $T_{\ell_1},\ldots,T_{\ell_r}$ with coefficients in $\z$.

By definition, the Hecke ring $\mathbf{T}^0_k$ is the subring of $\enn_\ci(\M_k^0)$ generated by all the operators
$T_n$, for $n>0$, and the Hecke algebra ${(\mathbf{T}^0_k)}_\ci$
is the smallest $\ci$--subalgebra of $\enn_\ci(\M_k^0)$ containing all the $T_n$'s. 

For every $n$, the operator $T_n$ is a semi--simple endomorphism preserving the integral structure $\M^0_k(\z)$, moreover the algebra
${(\mathbf{T}^0_k)}_\ci$ acts on $\M^0_k$ with multiplicity one. As a consequence of these two facts one has

\begin{teo}\label{Heckering} There exist number fields $K_i$, for $1\leq i\leq r$, with rings of integers $O_i$, and an injective ring homomorphism
$$\theta_k:\mathbf{T}^0_k\longrightarrow\prod_{1\leq i\leq r} O_i$$
which has finite cokernel. The rank of $\mathbf{T}^0_k$ as $\z$--module is equal
to $\dim_\ci(\M^0_k)$.
\end{teo}

A {\it system of eigenvalues} arising from $\M^0_k$ is a collection $(a_\ell)$ of complex numbers, indexed by all the primes $\ell$,
so that there exists a nonzero form $f\in\M^0_k$ for which $T_\ell(f)=a_\ell f$, for all $\ell$. One can show that there is a bijection
between systems of eigenvalues arising from $\M^0_k$ and ${\rm Hom}_{\rm rings}(\mathbf{T}^0_k,\ci)$.

If $\theta_{k,i}:\mathbf{T}^0_k\rightarrow O_i$ denotes the composition of $\theta_k$ with the projection onto $O_i$, then
all the systems of eigenvalues arising from $\M^0_k$ are described by $(\sigma(\theta_{k,i}(T_\ell)))$, where $1\leq i\leq r$ and
$\sigma\in\gq$ is any element (each $K_i$ is considered as a subfield of $\ci$).

Let us remark that in any known example $r$ is equal to $1$ and the systems of eigenvalues arising from $\M^0_k$ form a unique Galois orbit.
Maeda's conjecture is the statement that this happens for all $k$.

The Hecke ring $\mathbf{T}^0_k$ acts naturally on the space $\M^0_k(\f_p)$ and, by extension of scalars, on $\M^0_k(\f_p)\otimes\fbar$, simply
denoted by $\M^0_k(\fbar)$ in what follows. A system of eigenvalues mod $p$ arising from $\M^0_k(\fbar)$ is a collection $\Phi={(a_\ell)}_{\ell\neq p}$
of elements $a_\ell\in\fbar$, indexed by primes $\ell\neq p$, so that there exists a nonzero form $f\in\M^0_k(\fbar)$ with $T_\ell(f)=a_\ell f$.

If $\Phi={(a_\ell)}_{\ell\neq p}$ is any system of eigenvalues mod $p$, one can find a nonzero form $f\in\M^0_k(\fbar)$ giving rise to
$\Phi$ that is an eigenvector for $T_p$. Therefore there is a ring homomorphism $\lambda_\Phi:\mathbf{T}^0_k\rightarrow\fbar$
defined by $T(f)=\lambda_\Phi(T)f$, for $T\in\mathbf{T}^0_k$. The $p$--th eigenvalue $a_p$, and hence the morphism $\lambda_\Phi$, is not unique in
general, for this reason we had preferred to not include it in the definition of eigensystem mod $p$. However it can be shown that uniqueness holds
if the weight $k$ is not too large with respect to $p$:

\begin{prop}\label{kleq2p-1} If $k\leq 2p-1$ then there is a natural bijection between mod $p$ systems of eigenvalues arising from $\M^0_k(\fbar)$
and the set of $\fbar$--valued points of $\spe(\mathbf{T}^0_k)$.
\end{prop}

%The proof can be found in an extended version of this paper (cf. \cite{Ce} Proposition $4.3$).

By a classical result of Eichler, Shimura and Deligne, to any mod $p$ system of eigenvalues $\Phi$
one can attach a continuous, semisimple Galois representation
$$\rho_\Phi:\gq\longrightarrow\gl_2(\fbar),$$
which is odd, unramified outside $p$, and that is characterized by
the equalities
\begin{equation}\label{frob}\tr(\rho_\Phi(\fr_\ell))=a_\ell,\hphantom{x}
\det(\rho_\Phi(\fr_\ell))=\ell^{k-1},
\end{equation}
for all primes $\ell\neq p$, where $\fr_\ell$ is a Frobenius element of $\gq$ at $\ell$.

If $h\in\z_{\geq 0}$ is a nonnegative integer then it follows from the theory of the
$\theta$--operator on mod $p$ modular forms that the collection ${(\ell^{h}a_\ell)}_{\ell\neq p}$
is a system of eigenvalues arising from $\M^0_{k+h(p+1)}(\fbar)$, that will be denoted by $\Phi^{(h)}$.
We have
$$\rho_{\Phi^{(h)}}\simeq\chi_p^h\otimes\rho_\Phi,$$
where $\chi_p:\gq\rightarrow\f_p^*$ is the mod $p$ cyclotomic character, and $\Phi^{(h)}$
is usually called the $h$--fold twist of $\Phi$.

The following theorem is due to Tate and Serre. It has been generalized to higher levels
by Jochnowitz (cf. \cite{Jo1}) and Ash--Stevens (cf. \cite{AS}).

\begin{teo}\label{tatetwist} If $\Phi$ is a system of mod $p$ eigenvalues arising from $\M^0_k(\fbar)$,
then there exists a twist $\Phi^{(h)}$ that arises from $\M^0_{k'}(\fbar)$, where
$2\leq k'\leq p+1$.
\end{teo}

In this weight range, and when $\rho_\Phi$ is irreducible, two theorems of Deligne and Fontaine say that
the semisimplification of local representation ${(\rho_\Phi)}_p$, obtained by restricting $\rho_\Phi$ to a decomposition subgroup
$D_p<\gq$ at $p$, is determined by the (unique) $a_p$ eigenvalue associated to $\Phi$ (cf. \cite{Ed}).
We only point out that $a_p\neq 0$ if and only if ${(\rho_\Phi)}_p$ is reducible.

Let $\Phi$ be a system of eigenvalues mod $p$, and assume that $\rho_\Phi$ is irreducible. Since we
are working with modular forms of level one, the local representation ${(\rho_\Phi)}_p$ {\it is} ramified and
one observes that ${(\rho_\Phi)}_p$ is semisimple if and only if it is tamely ramified. There is a criterion
for deciding when this happens.

\begin{teo}\label{tamramtwist} Let $\Phi$ be a system of eigenvalues arising from $\M^0_k(\fbar)$, where $2\leq k\leq p+1$,
and so that $\rho_\Phi$ is irreducible. Then ${(\rho_\Phi)}_p$ is tamely ramified if and only if one of the
following holds

i) $\Phi^{(2-k)}$ arises from $\M^0_{p+3-k}(\fbar)$;

ii) $\Phi^{(1-k)}$ arises from $\M^0_{p+1-k}(\fbar)$.
\end{teo}

From the description of ${(\rho_\Phi)}_p$ given by the theorems of Deligne and Fontaine, and from an elementary
analysis of the $\theta$--cycle of $\Phi$ (cf. \cite{Jo}), one sees that condition i) in the theorem is equivalent to ${(\rho_\Phi)}_p$
be irreducible. In the hardest case when $({\rho_\Phi})_p$ is reducible, the criterion was conjectured by Serre and proved
by Gross (cf. \cite{Gr}).

\section{Computations}\label{computations}

Let $p$ be any prime number, and let $\mathcal{E}^{\rm Irr}(p)$ be the set of all systems $\Phi={(a_\ell)}_{\ell\neq p}$ of Hecke
eigenvalues mod $p$ arising from $\M^0_k(\fbar)$, for some $k$, so that the associated Galois representation $\rho_\Phi$ is irreducible.
By the level one case of Serre's Conjecture, proved by Khare in $2005$ (cf. \cite{Kh}), the cardinalty of $\mathcal{E}^{\rm Irr}(p)$ is equal to
the integer $R(p)$ defined in the Introduction.

According to Theorem \ref{tatetwist}, any eigensystem $\Phi$ admits a twist in the weight range $2\leq k\leq p+1$. Since
the number of systems of eigenvalues mod $p$ arising from $\M^0_k$ is bounded from above by
$\dim_{\fbar}(\M^0_k(\fbar))=\dim_\ci(\M^0_k)$, we have the following inequality
\begin{equation}\label{upbou} R(p)=|\mathcal{E}^{\rm Irr}(p)|\leq(p-1)\sum_{2\leq k\leq p+1}\dim_\ci(\M^0_k).
\end{equation}
Let $U(p)$ be the upper bound for $R(p)$ given by the inequality above. Using the well--known formulas for $\dim_\ci(\M^0_k)$, one
finds that there is a degree $3$ polynomial $F_\alpha(x)\in\qu[x]$, depending only on the residue class $\alpha$ of $p$ mod $12$, and
unique if $p>3$, so that $F_\alpha(p)=U(p)$. Letting $p$ grow to infinity, one finds that
$$U(p)\sim p^3/48+O(p^2).$$
Professor Khare had raised the question of whether this estimate give the correct asymptotic behaviour with $p$ of $R(p)$
(cf. \cite{Kh1}, \S $8$), in our thesis we were led to conjecture a positive answer to his question.
The difficulty of this conjecture is in producing lower bounds for $R(p)$. In this direction, the best result known today is due to Serre, who showed
in an unpublished correspondence with Khare that $R(p)$ is bounded from below by a function of the type $cp^2+O(p)$, for a constant $c>0$.

In our computation, for all $p\leq 1999$, we obtain a lower bound $L(p)$ for $R(p)$ which is displayed in the table at the end of
the paper together with $U(p)$ and with the ratio $(U(p)-L(p))/p^2$. In the range explored the ratio $(U(p)-L(p))/p^2$ is close to
zero, putting in evidence the tendency for $R(p)$ to approach $U(p)$.

We are going to explain in details how we computed $L(p)$. The theoretical basis of the method is provided by the commutative algebra
explained in the last two sections of this paper. We adopt some of the notation established there. So that, for example, $\delta_R$ denotes
the discriminant of a finite $S_\qu$--ring $R$ (cf. section \ref{discr}).

Let $k$ be an even integer in the range $\{2,\ldots, p+1\}$ and let $\mathcal{E}(p,k)$ be the set of mod $p$ systems of Hecke
eigenvalues $\Phi$ appearing in the space $\M_k^0(\fbar)$.
Consider the following subsets of $\mathcal{E}(p,k)$, defined in terms of the
Galois representations $\rho_\Phi$ associated to $\Phi$.

$\mathcal{E}^{\rm Eis}(p,k)=\{\Phi\in\mathcal{E}(p,k)\hphantom{.}|
\hphantom{.}\rho_\Phi\hphantom{x}{\rm is}\hphantom{x}{\rm reducible}\};$

$\mathcal{E}^{p-\rm tame}(p,k)=\{\Phi\in\mathcal{E}(p,k)-\mathcal{E}^{\rm Eis}(p,k)
\hphantom{.}|\hphantom{.}{(\rho_\Phi)}_p\hphantom{x}
{\rm is}\hphantom{x}{\rm tamely}\hphantom{x}{\rm ramified}\};$

$\mathcal{E}^{p-\rm wild}(p,k)=\{\Phi\in\mathcal{E}(p,k)-\mathcal{E}^{\rm Eis}(p,k)
\hphantom{.}|\hphantom{.}{(\rho_\Phi)}_p\hphantom{x}
{\rm is}\hphantom{x}{\rm wildly}\hphantom{x}{\rm ramified}\};$

$\mathcal{E}^{p-\rm{split}}(p,k)=\{\Phi\in\mathcal{E}(p,k)-\mathcal{E}^{\rm Eis}(p,k)
\hphantom{.}|\hphantom{.}{(\rho_\Phi)}_p\hphantom{x}
{\rm is}\hphantom{x}{\rm decomposable}\};$

$\mathcal{E}^{p-\rm{irr}}(p,k)=\{\Phi\in\mathcal{E}(p,k)-\mathcal{E}^{\rm Eis}(p,k)
\hphantom{.}|\hphantom{.}{(\rho_\Phi)}_p\hphantom{x}{\rm is}
\hphantom{x}{\rm irreducible}\}.$

Notice that there are the following disjoint unions (cf. section \ref{gener}):
$$\mathcal{E}(p,k)=\mathcal{E}^{\rm Eis}(p,k)\cup\mathcal{E}^{p-\rm tame}(p,k)
\cup \mathcal{E}^{p-\rm wild}(p,k),$$
$$\mathcal{E}^{p-\rm tame}(p,k)=\mathcal{E}^{p-\rm{split}}(p,k)\cup
\mathcal{E}^{p-\rm{irr}}(p,k);$$

and, for $k\leq p+1$, there are natural bijections
$$\mathcal{E}^{p-\rm irr}(p,k)\ni\Phi\longleftrightarrow\Phi^{(2-k)}\in\mathcal{E}^{p-\rm irr}(p,p+3-k),$$
$$\mathcal{E}^{p-\rm split}(p,k)\ni\Phi\longleftrightarrow\Phi^{(1-k)}\in\mathcal{E}^{p-\rm split}(p,p+1-k).$$

From Theorem \ref{tamramtwist} we deduce the formula
$$|\mathcal{E}^{\rm Irr}(p)|=(p-1)\sum_{2\leq k\leq p+1}\left[|\mathcal{E}(p,k)|-|\mathcal{E}^{\rm Eis}(p,k)|-\dfrac{1}{2}|\mathcal{E}^{p-\rm{tame}}(p,k)|\right]$$

In order to estimate $|\mathcal{E}^{\rm Irr}(p)|=R(p)$ from below we compute, for $2\leq k\leq p+1$, the values of $|\mathcal{E}(p,k)|$ and $|\mathcal{E}^{\rm Eis}(p,k)|$, and an
upper bound for $|\mathcal{E}^{p-\rm{tame}}(p,k)|$.

\subsection{Computation of $|\mathcal{E}^{\rm Eis}(p,k)|$}

This is the simplest quantity to compute, at least when $k\leq p+1$,
thank to the following criterion.
\begin{prop} Let $p$ be a prime and $k\leq p+1$ an integer so that
$\M_k^0(\f_p)\neq 0$. Then $\mathcal{E}^{\rm Eis}(p,k)$ is not empty
if and only if $p$ divides the numberator of the $k$--th Bernoulli
number $b_k$. Moreover, if $\mathcal{E}^{\rm Eis}(p,k)$ is not empty then
it consists only of the mod $p$ eigensystem $\Phi(E_k)={(1+\ell^{k-1})}_{\ell\neq p}$.
\end{prop}
\begin{proof} A possible proof can be carried out using a filtration
argument. The details can be found in (\cite{Se6}, \S $3.2$ i)), where
a proof in the case $k< p-1$ is given. The proof there extends to the
cases $k\leq p+1$, mainly thank to the fact that $\M_p^0(\f_p)=0$.
\end{proof}

\subsection{Computation of $|\mathcal{E}(p,k)|$}

Let $k$ be a weight $\leq p+1$, and $n_k$ the integer $\dim_{\fbar}(\M^0_k(\fbar))=\dim_\ci(\M^0_k)$.
Instead of computing directly $|\mathcal{E}(p,k)|$, we find convenient to compute the difference
$n_k-|\mathcal{E}(p,k)|$ between the number of characteristic zero eigensystems arising from $\M^0_k$ and
that of mod $p$ eigensystems arising from the same space. Such integer can be consider as a measure
of the occurrence of mod $p$ congruences between eigenforms in $\M^0_k$. 
The method used is described in the following application of proposition \ref{corp}.
\begin{prop}\label{crit1} Let $r$ be an integer $>0$, $T_r\in\mathbf{T}^0_k$ the $r$--th Hecke operator, and $h_r(x)\in\z[x]$ its characteristic
polynomial as an endomorphism of $\M_k^0(\ci)$. Assume that the discriminant $\delta_r$ of $h_r(x)$ is nonzero. Let $f_p^{(r)}$ be the number of
$\fbar$--valued points of the spectrum of the ring $\z[T_r]=\z[x]/(h_r(x))$, then
\begin{equation}\label{inequ1}|\mathcal{E}(p,k)|\geq f_p^{(r)}\geq n_k-\nu_p(\delta_r).
\end{equation}
Moreover if $f_p^{(r)}=n_k-\nu_p(\delta_r)$, then
\begin{equation}\label{eqcrit1}
|\mathcal{E}(p,k)|=f_p^{(r)}=n_k-\nu_p(\delta_r).
\end{equation}
In this case $p$ does not divide the index of $\z[T_r]$ in its integral closure inside $\z[T_r]\otimes\qu=\mathbf{T}^0_k\otimes\qu$.
In particular, $p$ does not divide $[\mathbf{T}^0_k:\z[T_r]]$, we have $\nu_p(\delta_{\mathbf{T}^0_k})=\nu_p(\delta_r)$,
and the inclusion $\z[T_r]\subset\mathbf{T}^0_k$ induces an isomorphism
$$\f_p[x]/(\bar h_r(x))\simeq\mathbf{T}^0_k/p\mathbf{T}^0_k,$$
where $\bar h_r(x)$ denotes the reduction mod $p$ of $h_r(x)$.
\end{prop}

Notice that the integer $f^{(r)}_p$ is simply the degree of the largest square--free factor of the reduction mod $p$ of $h_r(x)$.

As stated in the proposition, the subring $\z[T_r]\subset\enn_\ci(\M^0_k)$ is isomorphic to $\z[x]/(h_r(x))$ thank to the fact that
$T_r$ is a semisimple endomorphism of $\M^0_k$, and to the assumption $\delta_r\neq 0$.

\begin{defi} If the characteristic polynomial $h_r(x)$ of $T_r$ acting on $\M^0_k$ has nonzero discriminant and
satisfies the numerical condition
$$f_p^{(r)}=n_k-\nu_p(\delta_r)$$
appearing in second part of the proposition, then we will say that the Hecke operator $T_r$, acting on $\M^0_k$, is $p$--good.
\end{defi}

Of course the proposition can only be useful if one disposes of an Hecke operator $T_r$ so that $\delta_r\neq 0$,
which amounts to the requirement that the eigenvalues of $T_r$ acting on $\M^0_k$ be pairwise distinct. This condition is
perhaps not too restrictive since in all known cases $h_r(x)$ is even {\it irreducible}, for $r>1$.

Consider all pairs $(p,k)$, where $p$ is a prime number $\leq 1999$, and $k$ is an even integer $\leq p+1$ so that $\M^0_k$ is nonzero.
For each such pair, we had looked for the least integer $r$, with $1<r<20$, so that $T_r$ acting on $\M^0_k$ is a $p$--good Hecke operator.
In the table below we describe for how many pairs $(p,k)$ a given $r$ with $1<r<20$ had such property.

\vspace{5mm}

\begin{center}

\begin{tabular}{c|cccccccccc}
r	&2		&3		&5	&6	&7	&10	&11	&12	&14	&17\\
\hline
	&136611	&205	&28	&1	&10	&2	&4	&2	&1	&1\\
\end{tabular}

\end{center}

\vspace{5mm}

Out of the $136873$ many pairs $(p,k)$ considered, only in $8$ cases there is no integer $r<20$ (and there seem to be no integer at all)
so that $T_r$ acting on $\M^0_k$ is $p$--good. It is the ease of finding $p$--good Hecke operators that makes Proposition \ref{crit1}
efficient for computing the difference $n_k-|\mathcal{E}(p,k)|$. We found that $n_k-|\mathcal{E}(p,k)|$ is always $<3$, and the number of
times the values $0, 1$ and $2$ are attained are described by the next table, which gives an idea of how rare congruences are in this setting.

\vspace{5mm}

\begin{center}

\begin{tabular}{c|ccc}
$t$						&0		&1		&2\\
\hline
$|\left(n_k-|\mathcal{E}(p ,k)|\right)^{-1}(t)|$	&135703	&161	&1\\
\end{tabular}

\end{center}

\vspace{5mm}

The $8$ pairs $(p,k)$ for which we are unable to find a $p$--good Hecke operator acting on $\M^0_k$ are: $(491, 246)$,  $(563, 282)$, $(751, 376)$, $(1399, 700)$,
$(1423, 712)$, $(1823, 912)$, $(1879, 940)$, $(1931, 916)$. All these pairs are of the form $(p,(p+1)/2)$, and the space $\M^0_{(p+1)/2}$ gives rise to a set of mod $p$
systems of eigenvalues whose associated representations are of dihedral type. We have a good understanding of dihedral systems, and in subsection
\ref{dihecase} we explain how we computed $|\mathcal{E}(p,k)|$ in these cases. As it turns out, in all the $8$ cases $|\mathcal{E}(p,k)|=n_k$.

%\begin{rem} An empirical observation is that for the large majority of the pairs $(p,k)$, where $p$ is a prime $<3000$, and $k$ is an even weight $\leq p+1$
%and so that $\M^0_k(\fbar)$ is nonzero, it is really easy to find an Hecke operator $T_\ell$ for which the numerical condition
%$f_p^{(\ell)}=n_k-\nu_p(\delta_{h_\ell})$ is satisfied. [Add more precise info here]. In fact we failed to find such a $T_\ell$ only in $5$ cases. In all of them
%the weight $k$ is equal to $(p+1)/2$ and we believe that this failure is due to the existence of systems of eigenvalues $\Phi$ so that $\rho_\Phi$ has small image.
%We get around this problem by exploiting the fact that the space $\M^0_{(p+1)/2}(\fbar)$ gives rise to system of eigenvalues $\Phi$ of a very simple nature, namely
%those for which $\rho_\Phi$ is {\it dihedral}. We will explain how we computed $|\mathcal{E}(p,k)|$ in these cases in subsection \ref{dihecase}, as it turns out
%$|\mathcal{E}(p,k)|$ is equal to $n_k$, the maximum possible value.
%\end{rem}

%a good amount of information can already be recovered from the study of the characteristic polynomial of a single Hecke operator 

\begin{rem} Let $\tilde{\mathbf{T}}^0_k$ be the integral closure of the Hecke ring $\mathbf{T}^0_k$ in $\mathbf{T}^0_k\otimes\qu$.
For all the pairs $(p,k)$ considered, $p$ does not divide the index of $\mathbf{T}^0_k$ in $\tilde{\mathbf{T}}^0_k$.
This follows from proposition \ref{crit1} whenever there exists a $p$--good Hecke operator $T_r$ acting on $\M^0_k$, and it follows
from the equality $|\mathcal{E}(p,k)|=n_k$ in the remaining $8$ cases. The conclusion is that, if $k\leq p+1$ and $p\leq 1999$, we have
$\Hom_{\rm rin gs}(\mathbf{T}^0_k,\fbar)=\Hom_{\rm rings}(\tilde{\mathbf{T}}^0_k,\fbar)$, and there is no example of a mod $p$ congruence
between two distinct eigensystems arising from $\M^0_k$ caused by the fact that the order $\mathbf{T}^0_k$ is not maximal at $p$. In other words, all
the mod $p$ congruences between distinct characteristic zero Hecke eigensystems arising from $\M^0_k$ that we had found can be explained in terms of
ramification properties above $p$ of the components of $\mathbf{T}^0_k\otimes\qu$.
%, which are the number fields generated by the $\gq$--orbits of eigensystems arising from $\M^0_k$..
\end{rem}

\subsection{An upper bound for $|\mathcal{E}^{p-{\rm tame}}(p,k)|$}\label{comptam}
The set $\mathcal{E}^{p-{\rm tame}}(p,k)$ is the disjoint union of
$\mathcal{E}^{p-{\rm split}}(p,k)$ and $\mathcal{E}^{p-{\rm irr}}(p,k)$, and
we will bound these two sets separately. In order to bound the size of
$\mathcal{E}^{p-{\rm split}}(p,k)$ (resp. $\mathcal{E}^{p-{\rm irr}}(p,k)$) we need to estimate how often there exists a
system of eigenvalues $\Phi$ arising from $\M_k^0(\fbar)$ so that the eigensystem
$\Phi^{(1-k)}$ (resp. $\Phi^{(2-k)}$) arises from $\M_{p+1-k}^0(\fbar)$ (resp. $\M_{p+3-k}^0(\fbar)$) (cf. Theorem \ref{tamramtwist}).

Let $h(x)$ and $j(x)$ be monic polynomials in $\z[x]$ and let $p$ be any prime
number. Consider the greatest common divisor $d_p(x)\in\f_p[x]$ of the reduction
mod $p$ of $h(x)$ and $j(x)$.

\begin{defi} The {\sl linking number at $p$} of $h(x)$ and $j(x)$ is the degree
of $d_p(x)$, it is denoteb by $e_p(h,j)$.
\end{defi}

The integer $e_p(h,j)$ is a measure of the congruences mod $p$ between the roots of $h(x)$
and $j(x)$. It is zero if and only if the reduction mod $p$ of $h(x)$ and $j(x)$ have no common roots
$\fbar$.

\begin{prop}\label{tamram} Let $\ell\neq p$ be any prime, $h(x)\in\z[x]$ the characteristic polynomial of
$T_\ell$ acting on $\M^0_k$, and $j(x)\in\z[x]$ the characteristic polynomial of $\ell^{k-1}T_\ell$ acting on $\M^0_{p+1-k}$.
Then
$$|\mathcal{E}^{p-{\rm split}}(p,k)|\leq e_p(h,j).$$
\end{prop}
\begin{proof} Let $\Phi$ be a system of eigenvalues arising from $\M_k^0(\fbar)$ so that
$\rho_\Phi$ is irreducible. By the tameness criterion established by Gross,
the restriction of $\rho_\Phi$ to a decomposition group at $p$ is decomposable if and only
if there exists a system of mod $p$ eigenvalues $(b_\ell)$ arising from
$\M_{p+1-k}^0(\fbar)$ so that
$$a_q=q^{k-1}b_q,$$
for all primes $q\neq p$. In particular, setting $q=\ell$, we see that
$$|\mathcal{E}^{p-{\rm split}}(p,k)|\leq e_p(h,j),$$
where $e_p(h,j)$ is the linking number at $p$ of the polynomials $h(x)$ and $j(x)$.
The proposition follows.
\end{proof}

Similarly we have

\begin{prop}\label{tamram2} Let $\ell\neq p$ be any prime, $h(x)\in\z[x]$ the characteristic polynomial of
$T_\ell$ acting on $\M^0_k$, and $j(x)\in\z[x]$ the characteristic polynomial of $\ell^{k-2}T_\ell$ acting on $\M^0_{p+3-k}$.
Then
$$|\mathcal{E}^{p-{\rm irr}}(p,k)|\leq e_p(h,j).$$
\end{prop}
%\begin{proof} The proof of the proposition is analogous to that of proposition
%\ref{tamram}. One has to keep in mind that if $2\leq k\leq p+1$ then a system of eigenvalues
%$\Phi\in\mathcal{E}(p,k)$ belongs to $\mathcal{E}^{p-{\rm irr}}(p,k)$ if and only if
%the system $\Phi^{2-k}$ appears in the space $\M_{p+3-k}^0(\fbar)$ (cf. remark \ref{sstw}).
%\end{proof}

For any given prime $\ell\neq p$, the two propositions provide an upper bound for $|\mathcal{E}^{p-{\rm split}}(p,k)|$ and
$|\mathcal{E}^{p-{\rm irr}}(p,k)|$. We computed these estimates for $\ell=2$ and $3$ and kept the smallest values so obtained.
In the special case where $k=(p+1)/2$ (resp. $k=(p+3)/2$), in order to bound $|\mathcal{E}^{p-{\rm split}}(p,k)|$ (resp.
$|\mathcal{E}^{p-{\rm irr}}(p,k)|$) we used the smallest prime $\ell\neq p$ that is not a quadratic residue mod $p$, for otherwise
the characteristic polynomials of $T_\ell$ and $\ell^{k-1}T_\ell$ (resp. $\ell^{k-2}T_\ell$) acting on $\M^0_k$
would have the same mod $p$ reduction and the resulting upper bound would be $\dim_\ci(\M^0_k)$, the worse possible.

When $p\equiv 3$ mod $4$ and $k=(p+1)/2$, if $h$ is the class number of $\qu(\sqrt {-p})$, it can be shown that the set
$\mathcal{E}^{p-{\rm split}}(p,k)$ contains precisely $(h-1)/2$ eigensystems $\Phi$ so that $\Phi=\Phi^{(p-1)/2}$. These are the
eigensystems whose associated representations are of dihedral type (cf. subsection \ref{dihecase}). Let $\mathcal{E}^{p-{\rm split}, {\rm nd}}(p,k)$ be the subset
of $\mathcal{E}^{p-{\rm split}}(p,k)$ of all the eigensystems $\Phi$ so that $\Phi\neq\Phi^{(p-1)/2}$. Out of the $136873$ pairs $(p,k)$
considered, we found that $\mathcal{E}^{p-{\rm split}, {\rm nd}}(p,k)$ is empty in $136706$ cases, furthermore its cardinality never exceeds $3$.

Of the $299$ many primes $p$ with $11\leq p\leq 1999$, we found that for $224$ many of them there are no mod $p$ Galois representations
that are irreducible, odd, $2$--dimensional, unramified outside $p$, non--dihedral and so that the local representation at $p$ is decomposable.

For what concerns the other class of tamely ramified representations, we report that $|\mathcal{E}^{p-\rm{irr}}(p,k)|$ is zero
in $136696$ cases, and it is always $<5$. In the range considered, for $218$ many primes $p$ there are no mod $p$ Galois
representations of the type considered so that the local representation at $p$ is irreducible.

If $p$ is a prime for which we know that $\mathcal{E}^{p-{\rm split}, {\rm nd}}(p,k)$ and $\mathcal{E}^{p-\rm{irr}}(p,k)$ are empty for all
$k\leq p+1$, then our method lead to the exact value of $R(p)$. This happens for $164$ many primes, they appear in the table with the
symbol $\sp*$ typed next to the corresponding value $L(p)$ in the second column.

\subsection{The dihedral case}\label{dihecase} Let $\Phi$ be a system of mod $p$ eigenvalues
arising from $\M^0_k(\fbar)$ so that $\rho_\Phi$ is of {\sl dihedral type}, meaning
that the projective image $G$ of $\rho_\Phi$ in $\pgl_2(\fbar)$ is isomorphic
to a dihedral group $C_n\rtimes\z/2\z$, where $C_n$ is a cyclic group
of order $n\geq 2$ and the nontrivial element of $\z/2\z$ acts on $C_n$ by inversion.
Since $\rho_\Phi$ is, by definition, semisimple, it follows that any representation
$\rho_\Phi$ of dihedral type acts on $\fbar^2$ irreducibly.

Representations of dihedral type fit in the class of ``small--image'' representations
and are among them the easiest to understand and classify. It can be shown that

\begin{prop} Let $\Phi$ be an eigensystem arising from $\M^0_k(\fbar)$, with
$2\leq k\leq p+1$. The representation $\rho_\Phi$ is of dihedral type if and only if $\Phi=\Phi^{(p-1)/2}$. In this case we have

i) $\rho_\Phi$ is tamely ramified at $p$;

ii) $p\equiv 3$ mod $4$, $k=(p+1)/2$;

iii) the local representation ${(\rho_\Phi)}_p$ is described by the sum
of the trivial character and the quadratic character $\chi_p^{(p-1)/2}$,
where $\chi_p$ denotes the mod $p$ cyclotomic character of $G_p=G(\overline\qu_p/\qu_p)$;

iv) the image of $\rho_\Phi$ is isomorphic to $C_n\rtimes\z/2\z$, with $n$ odd;

v) $\rho_\Phi={\rm Ind}_K^Q(\Psi)$, where $K=\qu{(\sqrt{-p})}$, and $\Psi:G_K\rightarrow\fbar^*$ is a
continuous, everywhere unramified character.
\\Furthermore, there are precisely $(h-1)/2$ distinct isomorphism classes
of such $\rho_\Phi$, where $h$ is the class number of the imaginary
quadratic field $\qu(\sqrt{-p})$.
\end{prop}

The last statement of the proposition is essentially a modularity result for dihedral representations.
This case of Serre's Conjecture was known much earlier thank to the work of Hecke.

\begin{rem}
For a prime $p\equiv 3$ mod $4$, a consequence of the proposition is that if $\ell$ is a prime that is not a quadratic residue
mod $p$, then the characteristic polynomial $\bar h_\ell(x)\in\f_p[x]$ of the Hecke operator $T_\ell$ acting on $\M^0_k(\fbar)$ is divisible by
$x^{(h-1)/2}$. Using this simple fact we succeed in computing the value of $|\mathcal{E}(p,k)|$ in the few cases where we were not able to apply
the criterion of Proposition \ref{crit1}.
\end{rem}

\section{Discriminants of $S_\qu$--rings}\label{discr}

In the next two sections we describe the theoretical basis of our computations by working in an axiomatic setting.
In this section we introduce a special class of rings generalizing orders of number fields and recall definition and basic properties
of their discriminant.

\begin{defi} A ring $R$, commutative with identity, is called a {\sl finite $S_\qu$--ring}
if the following conditions are satisfied:

i) $R$ is finite and free as $\z$--module;

ii) $R\otimes\qu$ is isomorphic to a product of fields.
\\The {\sl rank} of $R$ is its rank as $\z$--module.
\end{defi}

Condition ii) can be replaced by

ii)' $R$ is reduced;
\\without affecting the notion just introduced. Our motivation for considering finite
$S_\qu$--rings is that the Hecke ring $\mathbf{T}^0_k$ is of this type (cf. Theorem \ref{Heckering}).

It is clear at once that if $R$ is a finite $S_\qu$--ring, and $R'\subset R$ is a subring
of finite index, then $R'$ is itself a finite $S_\qu$--ring of the same rank as $R$. Furthermore,
the product of finitely many finite $S_\qu$--rings is also a finite $S_\qu$--ring. If $h(x)\in\z[x]$
is a monic polynomial, then $R_h=\z[x]/(h(x))$ is a finite $S_\qu$--ring if and only if it
is reduced, i.e. if and only if $h(x)$ is square free.

Let $R$ be any finite $S_\qu$--ring of rank $n$, and regard it as a subring of $R\otimes\qu$ via
the injection $a\rightarrow a\otimes 1$.
The Artin ring $R\otimes\qu$ decomposes as the
product finitely many local Artin rings
$$R\otimes\qu\simeq\prod_{1\leq i\leq r}K_i,$$
and the factors of the decomposition are in correspondence with its prime
ideals. By assumption, every $K_i$ is a field, necessarily finite over $\qu$;
we have
$$n=\sum_{1\leq i\leq r}[K_i:\qu].$$
%
%Any ring homomorphis
%$$\sigma:R\otimes\qu\longrightarrow\overline\qu$$
%is uniquely determined by its restriction to $R$ and factors through exactly
%one of the projections $R\otimes\qu\rightarrow K_{i}$. Clearly
%any ring homomorphism $\sigma':R\rightarrow\qu$ extends uniquely to $R\otimes\qu$.

The ring extension $\z\subset R$ is finite and therefore integral. It follows
that the integral closure $\tilde R$ of $R$ in $R\otimes\qu$ coincides
with that of $\z$. Therefore, if $R_i$ denotes the ring
of integers of $K_i$, we see that
$$\tilde R=\prod_{1\leq i\leq r}R_i.$$
Moreover $R$ has finite index in $\tilde R$, since the ranks of both rings
equal to $\dim_\qu(R\otimes\qu)$. We have shown 

\begin{prop} Any finite $S_\qu$--ring $R$ is isomorphic to a finite index subring
of the product of the rings of integers $R_i$'s of finitely many number fields
$K_i$'s.
\end{prop}

The {\sl discriminant} $\delta_R$ of a finite $S_\qu$--ring $R$ is defined to be the
determinant of the bilinear form
$$R\times R\ni(x,y)\longrightarrow\tr(xy)\in\z,$$
where, for $a\in R$, $\tr(a)$ denotes the trace of the $\qu$--linear map
$$l_a:R\otimes\qu\longrightarrow R\otimes\qu$$
given by multiplication by $a\otimes 1$. It is easy to show that
\begin{equation}\tr(a)=\sum_{\sigma}\sigma(a),
\end{equation}
where the sum is extended to all the ring homomorphisms $\sigma:R\rightarrow\overline\qu$.

If $R$ is the ring of integers of a number field $K$, then $\delta_R$ coincides
with the discriminant $\delta_K$ of $K$.

The discriminant is multiplicative on any finite products of finite $S_\qu$--rings, and
if $R'\subset R$ is a subring of finite index $d$, then $\delta_{R'}=\delta_Rd^2$.
In particular $\delta_R\neq 0$ for any finite $S_\qu$--ring $R$, since $\delta_K\neq 0$
for any number field $K$. If $h(x)\in\z[x]$ is a monic, square free polynomial of
discriminant $\delta_h$, then $\delta_{R_h}=\delta_h$ (cf. \cite{Ma}, Theorem $8$).

\section{Discriminants and $\fbar$--valued points of $\spe (R)$}\label{secu}

The goal of this section is to prove Theorem \ref{prop} which, for a finite $S_\qu$--ring $R$, gives a lower
bound for the number of $\fbar$--valued points of $\spe (R)$, in terms of the
$p$--adic valuation of the discriminant of $R$.  We also obtain a criterion (Proposition \ref{corp}) which
gives a sufficient condition for the index of a monogenic subring $\z[T]\subset R$ to be prime to $p$. 

For a prime number $p$, let $\nu_p$ denote the additive $p$--adic valuation
of $\qu_p$, normalized so that $\nu_p(p)=1$.

\begin{lem}\label{lemu} Let $R$ be the ring of integers of a number field $K$ of degree $n$
over $\qu$ and of discriminant $\delta_K$. If $p$ is any prime, let $f_p$ be the number
of $\fbar$--valued points of $\spe (R)$. Then
$$f_p\geq n-\nu_p(\delta_K).$$
Moreover, equality holds if and only if $p$ is tamely ramified in $R$.
\end{lem}

\begin{proof} For a prime $\p$ of $K$ above $p$, let $f_\p$ and $e_\p$ denote, respectively,
the inertial degree and ramification index associated to $\p$. There is the well--known
formula (cf. \cite{Se}, I \S $5$, Prop. $10$)
\begin{equation}\label{localglobal} \sum_\p e_\p f_\p=n
\end{equation}
where the sum is extended to all the primes of $R$ above $p$.

Let $K_\p$ be the completion at $\p$ of $K$ and $\p^{r_\p}$ be the different of the local
extension $K_\p/\qu_p$. We know that
\begin{equation}\label{tamely} r_\p\geq e_\p-1,
\end{equation}
and equality holds if and only if $\p$ is tamely ramified (Serre, loc. cit.
III, \S $6$). The $p$--part of the discriminant $\delta_K$ is the product of the norms
of the fractional ideals $\p^{r_\p}$ of $K$, as $\p$ ranges among the prime ideals of $R$
above $p$ (Serre, loc. cit. III, \S $5$). Therefore we have
$$\nu_p(\delta_K)=\sum_\p f_\p r_\p.$$
Taking in account formula \ref{localglobal} and the inequality \ref{tamely}, we have
$$\sum_\p f_\p r_\p\geq \sum_\p f_\p (e_\p-1)=n-\sum_\p f_\p.$$
Moreover, equality holds if and only if every $\p$ is tamely ramified above $p$, that
is if and only if $p$ is tamely ramified in $K$. Observing that $\sum_\p f_\p=f_p$
concludes the proof of the lemma.
\end{proof}

We deduce two corollaries that follow from the proof of lemma \ref{lemu}.

\begin{cor}\label{coru} If $\nu_p(\delta_K)\leq p-1$ then $p$ is tamely ramified in $R$.
In particular $f_p=n-\nu_p(\delta_K)$.
\end{cor}
\begin{proof} Assume that $p$ is not tamely ramified in $K$, then there exists a prime
$\p_0$ of $R$ above $p$ so that $p|e_{\p_0}$ and, in the notation used in the proof of
lemma \ref{lemu}, $r_{\p_0}>e_{\p_0}-1$. In particular
$$r_{\p_0}>e_{\p_0}-1\geq p-1.$$
By the proof of lemma \ref{lemu}, we obtain
$$\nu_p(\delta_K)=\sum_\p f_\p r_\p>p-1,$$
which completes the proof of the corollary.
\end{proof}

\begin{cor}\label{cord} If $\nu_p(\delta_K)=1$ then there exists exactly one prime $\p_0$ of $R$
that lies above $p$ and that is ramified. We have $e_{\p_0}=2$, $f_{\p_0}=1$, and
$\spe (R)$ has exactly $n-1$ distinct $\fbar$--valued points.
\end{cor}
\begin{proof} By assumption $\nu_p(\delta_K)=1\leq p-1$, therefore corollary \ref{coru}
ensures that $p$ is tamely ramified in $R$. Applying lemma \ref{lemu} we obtain that
the number $f_p$ of distinct $\fbar$--valued points of $\spe (R)$ is
$$f_p=n-\nu_p(\delta_K)=n-1,$$
and the last part of the corollary follows. To see the first part, observe that $f_p$
is equal to the sum $\sum f_\p$ of the inertial degrees of the primes of $R$
of residual characteristic $p$. But since $f_p=n-1$, we easily see that
formula \ref{localglobal} forces the existence of exactly one ramified prime above
$p$, say $\p_0$, and for which, moreover, we must have $e_{\p_0}=2$ and $f_{\p_0}=1$.
\end{proof}

In order to prove theorem \ref{prop} we need the following

\begin{lem}\label{fini} Let $R'\subset R$ be an extension of finite $S_\qu$--rings so that
$R'$ has finite index $d$ in $R$. Let $f_p$ and $f_p'$ be the numbers of
$\fbar$--valued points of, respectively, $\spe (R)$ and $\spe (R')$. Then
$$f_p\geq f_p'\geq f_p-\nu_p(d).$$
\end{lem}
\begin{proof} The extension $R'\subset R$ is finite, therefore integral,
and any $\fbar$--valued point of $\spe (R')$ can be lifted to one of $\spe (R)$
(cf. \cite{AM} Theorem $5.16$), and the first inequality $f_p\geq f_p'$ readily
follows.

To see the other inequality, note that the inclusion $R'\subset R$ induces
an injective ring homorphism
$$\iota:R'/I\hookrightarrow R/pR,$$
where $I=pR\cap R'$ is the ideal of $R'$ given by the contraction of $(p)\subset R$,
and $R'/I$ may be identified with an $\f_p$--subalgebra of $R/pR$.

The cokernel of $\iota$ is an abelian group isomorphic to
$(R/R')/p(R/R')$, we have
$$|(R/pR)/(R'/I)|=|(R/R')/p(R/R')|\leq p^{\nu_p(d)}.$$
If $n$ (resp. $n'$) is the dimension of $R/pR$ (resp. $R'/I$) over
$\f_p$, then the previous inequality implies
$$n-n'\leq\nu_p(d).$$
Let $\sqrt{0}$ (resp. $\sqrt{0}'$) be the nilradical ideal of
$R/pR$ (resp. $R'/I$), and let ${\left(\vphantom{R'}R/pR\right)}_{{\rm red}}$
(resp. ${\left(R'/I\right)}_{{\rm red}}$) be the reduced ring associated to
$R/pR$ (resp. $R'/I$). We have the following exact sequences of
$\f_p$--vector spaces:
$$0\longrightarrow\sqrt{0}\longrightarrow R/p
\longrightarrow{\left(\vphantom{R'}R/pR\right)}_{{\rm red}}\longrightarrow 0,$$
$$0\longrightarrow\sqrt{0}'\longrightarrow R'/I
\longrightarrow{\left(R'/I\right)}_{{\rm red}}\longrightarrow 0.$$
Now, the injection $R'/I\hookrightarrow R/p$ induces the inclusions
$$\sqrt{0}'\subset\sqrt{0}\hphantom{x}\textrm{and}\hphantom{x}
{\left(R'/I\right)}_{{\rm red}}\subset {\left(\vphantom{R'}R/pR\right)}_{{\rm red}}.$$
Therefore there is a natural morphism between the exact sequences above, from the
lower to the upper one, described by three inclusions. If $r$ (resp. $r'$) is
the dimension of $\sqrt{0}$ (resp. $\sqrt{0}'$), then we have
$$f_p'+r'-n'=f_p+r-n=0,$$
since $r'\leq r$, we obtain
$$f_p'=f_p - (n-n')+(r-r')\geq f_p - \nu_p(d),$$
and this completes the proof of the lemma.
\end{proof}

Lemma \ref{lemu} generalizes as follows

\begin{teo}\label{prop} Let $R$ be a finite $S_\qu$--ring of rank $n$ as $\z$--module.
If $p$ is any prime number, let $f_p$ denote the number of $\fbar$--valued points
of $\spe (R)$. Then
$$f_p\geq n-\nu_p(\delta_R).$$
Moreover, equality holds if and only if the index of $R$ in its integral closure
$\tilde R$ in $R\otimes\qu$ is prime to $p$ and $p$ is tamely ramified in each component of $R\otimes\qu$.
\end{teo}
\begin{proof} By lemma \ref{lemu}, the inequality expressed by the proposition is
satisfied when $R$ is the ring of integers of a number field $K$. Note that the integers
$f_p$ and $\nu_p(\delta_R)$, viewed as functions of $R$, are additive with respect
to finite product of $S_\qu$--rings. Therefore the inequality
$$f_p\geq n-\nu_p(\delta_R)$$
holds for any finite $S_\qu$--ring $R$ that is isomorphic to a finite
product of rings of integers $R_i$ of number fields $K_i$, i.e. the inequality
of the proposition is proved for any finite $S_\qu$--ring $R$ that is integrally closed
in $R\otimes\qu$. In this case the second part of the proposition follows immediately
from Lemma \ref{lemu}.

Let now $R$ be any finite $S_\qu$--ring, let $\tilde R\subset R\otimes\qu$ be its
itegral closure, and let $d$ be the (finite) index $[\tilde R:R]$. If $\tilde f_p$
denote the number of $\fbar$--valued points of $\spe (\tilde R)$, then lemma
\ref{fini} applied to the extension $R\subset\tilde R$ says that
$$f_p\geq \tilde f_p-\nu_p(d).$$
We have seen that the proposition holds for $\tilde R$, therefore
$$f_p\geq n-\nu_p(\delta_{\tilde R})-\nu_p(d).$$
Since $\delta_R=\delta_{\tilde R}d^2$ we have
\begin{equation}\label{ineqdiscr}-\nu_p(\delta_{\tilde R})-\nu_p(d)\geq -\nu_p(\delta_R),
\end{equation}
and therefore
$$f_p\geq n-\nu_p(\delta_{R}),$$
completing the proof of the first part of the proposition. Now if $p$ divided $d$,
then inequality (\ref{ineqdiscr}) would certainly be strict and, consequently, $f_p$
would be strictly greater than $n-\nu_p(\delta_{R})$.
\end{proof}

The following proposition is a consequence of Theorem \ref{prop} and Lemma \ref{fini} and gives
a criterion for counting the number of $\fbar$--valued points of $\spe (R)$ in terms of numerical data encoded
in the characteristic polynomial of an element $T\in R$ that generates a finite index subring $\z[T]\subset R$.
It will be useful in our computations when $R$ is a Hecke ring $\mathbf{T}^0_k$ and
$T$ is an Hecke operator $T_\ell$.

\begin{prop}\label{corp} Let $R$ be a finite $S_\qu$--ring of rank $n$, $T\in R$ any element, and
$h(x)\in\z[x]$ its characteristic polynomial.
Assume hat the discriminant $\delta_h$ of $h(x)$ is nonzero.
Let $f_p$ be the number of $\fbar$--valued points of $\spe (R)$ and
$f_p^{(h)}$ that of the spectrum of $\z[T]=\z[x]/(h(x))$, then
$$f_p\geq f_p^{(h)}\geq n-\nu_p(\delta_h).$$
Moreover if $f_p^{(h)}=n-\nu_p(\delta_h)$, then
\begin{equation}\label{fundeq}
f_p= f_p^{(h)}= n-\nu_p(\delta_h).
\end{equation}
In this case $p$ does not divide the index $\z[T]$ in its integral closure in $\z[T]\otimes\qu=R\otimes\qu$. In particular,
$p$ does not divide the index $[R:\z[T]]$, we have $\nu_p(\delta_R)=\nu_p(\delta_h)$, and
the inclusion $\z[T]\subset R$ induces an isomorphism
$$\z[T]/p\z[T]\simeq R/pR.$$
\end{prop}

The characteristic polynomial $h(x)$ of $T\in R$ alluded to in the proposition is the monic characteristic polynomial
of the endomorphism of the $\qu$--vector space $R\otimes\qu$ given by multiplication by $T\otimes 1$.

 Notice that $f_p^{(h)}$ is simply the number of distinct roots in $\fbar$ of the reduction mod $p$ of $h(x)$, and $n$ is
the degree of $h(x)$. Thus the equality $f_p^{(h)}=n-\nu_p(\delta_h)$ is a numerical condition on $h(x)$.

\begin{proof} The ring $R$ is a finite $S_\qu$-ring and has no nilpotent elements. It follows that the endomorphism of
$R\otimes\qu$ given by multiplication by $T\otimes 1$ is semisimple, meaning that its minimal polynomial
is square free. Moreover, by assumption, the characteristic polynomial $h(x)$ of $T$ is square free and we conclude that
$h(x)$ is equal to the minimal polynomial of $T$. It follows that the subring $\z[T]$ has rank $n$ as an abelian group, hence
the index $[\z[T]:R]=d$ is finite.

Lemma \ref{fini} says that
$$f_p\geq f_p^{(h)}\geq f_p-\nu_p(d),$$
from which the the first part of the Proposition follows. Theorem \ref{prop} implies that
$$f_p\geq n-\nu_p(\delta_R),$$
and, since $\delta_h=\delta_rd^2$, putting together the two inequalities yields to
\begin{equation}\label{chainin} f_p\geq f_p^{(h)}\geq f_p-\nu_p(d)\geq n-\nu_p(\delta_R)-\nu_p(d)\geq n-\nu_p(\delta_h).
\end{equation}
Notice that the last inequality to the right is {\it strict} if $p$ divides $d$.

Now if $f_p^{(h)}=n-\nu_p(\delta_h)$, then the last three inequalities of (\ref{chainin}) are forced to be equalities. This immediately
implies that $\nu_p(d)=0$ and $f_p=f_p^{(h)}$, and we see that (\ref{fundeq}) of the Proposition holds.

To complete the proof of the Proposition we are only left with showing that $p$ does not divide the index of $\z[T]$ is its
integral closure, provided that the equality $f_p^{(h)}=n-\nu_p(\delta_h)$ holds. We had just shown that $p$ does not divide
the index $[R:\z[T]]$. Replacing $R$ by its integral closure $\tilde R$ and reasoning as above we easily see that $p$ does not
divide $[\tilde R:\z[T]]$, and the Proposition follows.
\end{proof}

\begin{rem} If there exists $T\in R$ so that $\nu_p(\delta_h)\leq 1$, then one knows that
the equality $f_p^{(h)}=n-\nu_p(\delta_h)$ is automatically satisfied. This is clear if $\nu_p(\delta_h)=0$, since in that
case the reduction mod $p$ of $h(x)$ is square free, and therefore $f_p^{(h)}=n$. In the case where $\nu_p(\delta_h)=1$,
we have that $h(x)$ has multiple roots when reduced mod $p$, therefore $n> f_p^{(h)}$. On the other hand,
by Theorem \ref{prop}, we have $f_p^{(h)}\geq n-\nu_p(\delta_h)=n-1$, therefore $f_p^{(h)}=n-1$ and the equality $f_p^{(h)}=n-\nu_p(\delta_h)$ holds.
In this last case, namely when $\nu_p(\delta_R)=1$, a complete description of the ramification of the components of $R\otimes\qu$
can be given: all of them but one are unramified above $p$, moreover the ramification above $p$ in the ramified component is that described
in corollary \ref{cord}.
\end{rem}

\newpage

\begin{sideways}

\begin{tabular}{cccc|cccc|cccc}

$p$ & $L$ & $U$ & $(U-L)/p^2$&$p$ & $L$ & $U$& $(U-L)/p^2$&$p$ & $L$ & $U$& $(U-L)/p^2$\\
\hline

11	&	10	*	&	10	&	0	&	83	&	10373	*	&	10414	&	0.005951	&	179	&	111784		&	112318	&	0.016666	\\
13	&	12	*	&	12	&	0	&	89	&	12848	*	&	12936	&	0.011109	&	181	&	115920	*	&	116100	&	0.005494	\\
17	&	48	*	&	48	&	0	&	97	&	16896	*	&	16896	&	0	&	191	&	136040		&	136990	&	0.02604	\\
19	&	108	*	&	108	&	0	&	101	&	19100	*	&	19200	&	0.009802	&	193	&	140928		&	141312	&	0.010309	\\
23	&	143	*	&	154	&	0.020793	&	103	&	22236	*	&	22440	&	0.019228	&	197	&	150528	*	&	150528	&	0	\\
29	&	336	*	&	336	&	0	&	107	&	22737		&	23002	&	0.023146	&	199	&	162756	*	&	163152	&	0.009999	\\
31	&	555	*	&	570	&	0.015608	&	109	&	24300	*	&	24300	&	0	&	211	&	194355	*	&	194460	&	0.002358	\\
37	&	720	*	&	756	&	0.026296	&	113	&	27104	*	&	27216	&	0.008771	&	223	&	229215		&	229770	&	0.01116	\\
41	&	1080	*	&	1080	&	0	&	127	&	42084	*	&	42210	&	0.007812	&	227	&	231424	*	&	232102	&	0.013157	\\
43	&	1554	*	&	1554	&	0	&	131	&	42510		&	43030	&	0.030301	&	229	&	237576		&	238260	&	0.013043	\\
47	&	1656	*	&	1702	&	0.020823	&	137	&	49368	*	&	49368	&	0	&	233	&	250792	*	&	251256	&	0.008546	\\
53	&	2496	*	&	2496	&	0	&	139	&	54717		&	55338	&	0.032141	&	239	&	270725	*	&	271558	&	0.014583	\\
59	&	3393		&	3538	&	0.041654	&	149	&	63788	*	&	63936	&	0.006666	&	241	&	277680		&	278400	&	0.012396	\\
61	&	3900	*	&	3900	&	0	&	151	&	70575		&	71100	&	0.023025	&	251	&	314875	*	&	315250	&	0.005952	\\
67	&	5940	*	&	6072	&	0.029405	&	157	&	74256	*	&	75036	&	0.031644	&	257	&	337664		&	338688	&	0.015503	\\
71	&	6195	*	&	6370	&	0.034715	&	163	&	89100	*	&	89424	&	0.012194	&	263	&	362084		&	363394	&	0.018939	\\
73	&	6840	*	&	6912	&	0.01351	&	167	&	90387	*	&	90802	&	0.01488	&	269	&	388332		&	389136	&	0.01111	\\
79	&	9906		&	10062	&	0.024995	&	173	&	100620		&	101136	&	0.01724	&	271	&	411345		&	412830	&	0.02022	\\

\end{tabular}

\end{sideways}

\newpage

\begin{sideways}
\begin{tabular}{cccc|cccc|cccc}

$p$ & $L$ & $U$ & $(U-L)/p^2$&$p$ & $L$ & $U$& $(U-L)/p^2$&$p$ & $L$ & $U$& $(U-L)/p^2$\\
\hline

277	&	425040		&	425316	&	0.003597	&	389	&	1190772	*	&	1191936	&	0.007692	&	499	&	2581383		&	2582628	&	0.004999	\\
281	&	444360	*	&	444360	&	0	&	397	&	1266804		&	1267596	&	0.005025	&	503	&	2590320		&	2593834	&	0.013888	\\
283	&	468825		&	470094	&	0.015844	&	401	&	1306000	*	&	1306800	&	0.004975	&	509	&	2688336	*	&	2688336	&	0	\\
293	&	503408	*	&	504576	&	0.013605	&	409	&	1386792	*	&	1387200	&	0.002439	&	521	&	2883400	*	&	2884440	&	0.003831	\\
307	&	598995		&	600372	&	0.01461	&	419	&	1491006		&	1492678	&	0.009523	&	523	&	2972268	*	&	2973834	&	0.005725	\\
311	&	602485		&	604810	&	0.024038	&	421	&	1513260		&	1514100	&	0.004739	&	541	&	3230820		&	3231900	&	0.00369	\\
313	&	616200		&	616512	&	0.003184	&	431	&	1623250		&	1625830	&	0.013888	&	547	&	3400761	*	&	3402672	&	0.006386	\\
317	&	640532	*	&	640848	&	0.003144	&	433	&	1646352		&	1648512	&	0.01152	&	557	&	3528376	*	&	3529488	&	0.003584	\\
331	&	751245		&	752730	&	0.013554	&	439	&	1755723		&	1758132	&	0.012499	&	563	&	3643446		&	3645694	&	0.007092	\\
337	&	771456	*	&	771456	&	0	&	443	&	1766232	*	&	1766674	&	0.002252	&	569	&	3763000		&	3764136	&	0.003508	\\
347	&	842164		&	843202	&	0.00862	&	449	&	1839040	*	&	1839936	&	0.004444	&	571	&	3869730	*	&	3870870	&	0.003496	\\
349	&	857472	*	&	857820	&	0.002857	&	457	&	1939824	*	&	1940736	&	0.004366	&	577	&	3924288		&	3926016	&	0.00519	\\
353	&	886336		&	888096	&	0.014124	&	461	&	1992260	*	&	1992720	&	0.002164	&	587	&	4132765	*	&	4135402	&	0.007653	\\
359	&	933127	*	&	934738	&	0.012499	&	463	&	2061213		&	2062830	&	0.007543	&	593	&	4263584	*	&	4264176	&	0.001683	\\
367	&	1025898	*	&	1026630	&	0.005434	&	467	&	2070205	*	&	2072302	&	0.009615	&	599	&	4389918		&	4395898	&	0.016666	\\
373	&	1049040		&	1049412	&	0.002673	&	479	&	2233216		&	2237518	&	0.018749	&	601	&	4438800		&	4440000	&	0.003322	\\
379	&	1128897		&	1130598	&	0.011842	&	487	&	2399625	*	&	2400840	&	0.005122	&	607	&	4648626	*	&	4651050	&	0.006578	\\
383	&	1135686	*	&	1137214	&	0.010416	&	491	&	2406880		&	2411290	&	0.018292	&	613	&	4712400	*	&	4713012	&	0.001628	\\

\end{tabular}

\end{sideways}

\newpage

\begin{sideways}
\begin{tabular}{cccc|cccc|cccc}

$p$ & $L$ & $U$ & $(U-L)/p^2$&$p$ & $L$ & $U$& $(U-L)/p^2$&$p$ & $L$ & $U$& $(U-L)/p^2$\\
\hline

617	&	4804184	*	&	4806648	&	0.006472	&	739	&	8394012	*	&	8395488	&	0.002702	&	859	&	13187031	*	&	13188318	&	0.001744	\\
619	&	4929786		&	4932258	&	0.006451	&	743	&	8414280		&	8419474	&	0.009408	&	863	&	13215322	*	&	13220494	&	0.006944	\\
631	&	5222070	*	&	5225220	&	0.007911	&	751	&	8806875		&	8811750	&	0.008643	&	877	&	13874964	*	&	13876716	&	0.002277	\\
641	&	5393280	*	&	5393280	&	0	&	757	&	8904924	*	&	8906436	&	0.002638	&	881	&	14066800	*	&	14068560	&	0.002267	\\
643	&	5527941		&	5528904	&	0.002329	&	761	&	9047800		&	9049320	&	0.002624	&	883	&	14324121		&	14325444	&	0.001696	\\
647	&	5541065		&	5547202	&	0.01466	&	769	&	9337344		&	9338880	&	0.002597	&	887	&	14352314	*	&	14359402	&	0.009008	\\
653	&	5701088		&	5703696	&	0.006116	&	773	&	9484792		&	9486336	&	0.002583	&	907	&	15524763		&	15526122	&	0.001651	\\
659	&	5861135	*	&	5863438	&	0.005303	&	787	&	10140186	*	&	10140972	&	0.001269	&	911	&	15553265		&	15561910	&	0.010416	\\
661	&	5914260	*	&	5916900	&	0.006042	&	797	&	10401332	*	&	10402128	&	0.001253	&	919	&	16145325		&	16151292	&	0.007065	\\
673	&	6245568	*	&	6246912	&	0.002967	&	809	&	10878912	*	&	10881336	&	0.003703	&	929	&	16504480	*	&	16506336	&	0.00215	\\
677	&	6357780	*	&	6359808	&	0.004424	&	811	&	11094165		&	11097810	&	0.005541	&	937	&	16937856	*	&	16937856	&	0	\\
683	&	6529468	*	&	6531514	&	0.004385	&	821	&	11373400	*	&	11375040	&	0.002433	&	941	&	17156880	*	&	17156880	&	0	\\
691	&	6859980	*	&	6862740	&	0.00578	&	823	&	11596776	*	&	11598420	&	0.002427	&	947	&	17487756	*	&	17488702	&	0.001054	\\
701	&	7063700		&	7064400	&	0.001424	&	827	&	11624711	*	&	11627602	&	0.004227	&	953	&	17822392	*	&	17824296	&	0.002096	\\
709	&	7309392	*	&	7310100	&	0.001408	&	829	&	11712060	*	&	11712060	&	0	&	967	&	18812367		&	18817680	&	0.005681	\\
719	&	7619057		&	7625878	&	0.013194	&	839	&	12133402		&	12143458	&	0.014285	&	971	&	18853405	*	&	18857770	&	0.004629	\\
727	&	7990356	*	&	7993260	&	0.005494	&	853	&	12762960		&	12763812	&	0.00117	&	977	&	19210608	*	&	19210608	&	0	\\
733	&	8080548		&	8082012	&	0.002724	&	857	&	12943576		&	12945288	&	0.00233	&	983	&	19558985		&	19568314	&	0.009654	\\

\end{tabular}
\end{sideways}

\newpage

\begin{sideways}
\begin{tabular}{cccc|cccc|cccc}

$p$ & $L$ & $U$ & $(U-L)/p^2$&$p$ & $L$ & $U$& $(U-L)/p^2$&$p$ & $L$ & $U$& $(U-L)/p^2$\\
\hline

991	&	20249460		&	20254410	&	0.00504	&	1103	&	27672873	*	&	27678934	&	0.004981	&	1237	&	39081084	*	&	39083556	&	0.001615	\\
997	&	20418996	*	&	20418996	&	0	&	1109	&	28134336	*	&	28134336	&	0	&	1249	&	40234272		&	40235520	&	0.000799	\\
1009	&	21164976		&	21168000	&	0.00297	&	1117	&	28747044		&	28749276	&	0.001788	&	1259	&	41206419		&	41213338	&	0.004365	\\
1013	&	21422016	*	&	21422016	&	0	&	1123	&	29474940		&	29477184	&	0.001779	&	1277	&	43008856		&	43011408	&	0.001564	\\
1019	&	21800470		&	21806578	&	0.005882	&	1129	&	29684448		&	29688960	&	0.003539	&	1279	&	43544655	*	&	43552962	&	0.005078	\\
1021	&	21935100	*	&	21935100	&	0	&	1151	&	31449050	*	&	31465150	&	0.012152	&	1283	&	43618127	*	&	43622614	&	0.002725	\\
1031	&	22580175	*	&	22588930	&	0.008236	&	1153	&	31627008	*	&	31629312	&	0.001733	&	1289	&	44237648	*	&	44238936	&	0.000775	\\
1033	&	22720512	*	&	22720512	&	0	&	1163	&	32459889	*	&	32462794	&	0.002147	&	1291	&	44782350		&	44790090	&	0.004643	\\
1039	&	23336835	*	&	23343582	&	0.006249	&	1171	&	33420465	*	&	33422220	&	0.001279	&	1297	&	45067104	*	&	45069696	&	0.00154	\\
1049	&	23795888		&	23796936	&	0.000952	&	1181	&	33992260		&	33998160	&	0.00423	&	1301	&	45485700	*	&	45489600	&	0.002304	\\
1051	&	24159450		&	24161550	&	0.001901	&	1187	&	34513786		&	34520902	&	0.00505	&	1303	&	46045881		&	46051740	&	0.00345	\\
1061	&	24622740		&	24625920	&	0.002824	&	1193	&	35047184	*	&	35048376	&	0.000837	&	1307	&	46118125	*	&	46124002	&	0.00344	\\
1063	&	24992577		&	24999480	&	0.006109	&	1201	&	35756400		&	35760000	&	0.002495	&	1319	&	47392644	*	&	47409778	&	0.009848	\\
1069	&	25187712		&	25188780	&	0.000934	&	1213	&	36846012	*	&	36846012	&	0	&	1321	&	47622960		&	47625600	&	0.001512	\\
1087	&	26728632		&	26731890	&	0.002757	&	1217	&	37208384	*	&	37213248	&	0.003284	&	1327	&	48638343	*	&	48644310	&	0.003388	\\
1091	&	26776940	*	&	26782390	&	0.004578	&	1223	&	37757967	*	&	37768354	&	0.006944	&	1361	&	52097520	*	&	52097520	&	0	\\
1093	&	26927628		&	26929812	&	0.001828	&	1229	&	38325880		&	38328336	&	0.001626	&	1367	&	52778142	*	&	52791802	&	0.007309	\\
1097	&	27224640		&	27227928	&	0.002732	&	1231	&	38820645		&	38829870	&	0.006087	&	1373	&	53486048		&	53491536	&	0.002911	\\

\end{tabular}
\end{sideways}

\newpage

\begin{sideways}
\begin{tabular}{cccc|cccc|cccc}

$p$ & $L$ & $U$ & $(U-L)/p^2$&$p$ & $L$ & $U$& $(U-L)/p^2$&$p$ & $L$ & $U$& $(U-L)/p^2$\\
\hline

1381	&	54429960		&	54434100	&	0.00217	&	1499	&	69649510	*	&	69658498	&	0.003999	&	1621	&	88134480	*	&	88136100	&	0.000616	\\
1399	&	56991567		&	57002052	&	0.005357	&	1511	&	71329380		&	71349010	&	0.008597	&	1627	&	89664957		&	89669022	&	0.001535	\\
1409	&	57820928	*	&	57822336	&	0.000709	&	1523	&	73062849	*	&	73066654	&	0.00164	&	1637	&	90775096		&	90778368	&	0.001221	\\
1423	&	59981382		&	59987070	&	0.002808	&	1531	&	74704545		&	74711430	&	0.002937	&	1657	&	94151880		&	94153536	&	0.000603	\\
1427	&	60066685		&	60073102	&	0.003151	&	1543	&	76473177		&	76483200	&	0.004209	&	1663	&	95744496		&	95756130	&	0.004206	\\
1429	&	60321576		&	60325860	&	0.002097	&	1549	&	76879872		&	76881420	&	0.000645	&	1667	&	95868304	*	&	95873302	&	0.001798	\\
1433	&	60835656	*	&	60835656	&	0	&	1553	&	77474288		&	77480496	&	0.002574	&	1669	&	96213576		&	96218580	&	0.001796	\\
1439	&	61588821		&	61605358	&	0.007986	&	1559	&	78363505	*	&	78384538	&	0.008653	&	1693	&	100438812	*	&	100438812	&	0	\\
1447	&	63065121	*	&	63074520	&	0.004488	&	1567	&	80103249	*	&	80108730	&	0.002232	&	1697	&	101152832		&	101154528	&	0.000588	\\
1451	&	63159100	*	&	63163450	&	0.002066	&	1571	&	80206590	*	&	80212870	&	0.002544	&	1699	&	102106683		&	102110928	&	0.00147	\\
1453	&	63423360		&	63424812	&	0.000687	&	1579	&	81958164		&	81962898	&	0.001898	&	1709	&	103315212		&	103320336	&	0.001754	\\
1459	&	64651365	*	&	64656468	&	0.002397	&	1583	&	82056758	*	&	82069414	&	0.00505	&	1721	&	105513400	*	&	105516840	&	0.001161	\\
1471	&	66256575	*	&	66266130	&	0.004415	&	1597	&	84262416		&	84270396	&	0.003128	&	1723	&	106495368		&	106500534	&	0.00174	\\
1481	&	67169800		&	67172760	&	0.001349	&	1601	&	84905600	*	&	84907200	&	0.000624	&	1733	&	107737328		&	107744256	&	0.002306	\\
1483	&	67895607		&	67900794	&	0.002358	&	1607	&	85857563	*	&	85868002	&	0.004042	&	1741	&	109245900	*	&	109245900	&	0	\\
1487	&	67980042	*	&	67994902	&	0.00672	&	1609	&	86185584	*	&	86188800	&	0.001242	&	1747	&	111008934	*	&	111014172	&	0.001716	\\
1489	&	68266464		&	68269440	&	0.001342	&	1613	&	86831992	*	&	86835216	&	0.001239	&	1753	&	111523560	*	&	111525312	&	0.00057	\\
1493	&	68822976	*	&	68822976	&	0	&	1619	&	87799961		&	87810478	&	0.004012	&	1759	&	113303979	*	&	113318922	&	0.004829	\\

\end{tabular}
\end{sideways}

\newpage

\begin{sideways}
\begin{tabular}{cccc|cccc}

$p$ & $L$ & $U$ & $(U-L)/p^2$&$p$ & $L$ & $U$& $(U-L)/p^2$\\
\hline

1777	&	116175264	*	&	116178816	&	0.001124	&	1913	&	145006080		&	145011816	&	0.001567	\\									
1783	&	118012950		&	118021860	&	0.002802	&	1931	&	149133030		&	149152330	&	0.005175	\\									
1787	&	118144793		&	118156402	&	0.003635	&	1933	&	149612148	*	&	149616012	&	0.001034	\\									
1789	&	118546188	*	&	118553340	&	0.002234	&	1949	&	153366040		&	153369936	&	0.001025	\\									
1801	&	120958200		&	120960000	&	0.000554	&	1951	&	154611600	*	&	154633050	&	0.005635	\\									
1811	&	122974115		&	122991310	&	0.005242	&	1973	&	159108848		&	159116736	&	0.002026	\\									
1823	&	125433768		&	125457454	&	0.007127	&	1979	&	160561183	*	&	160576018	&	0.003787	\\									
1831	&	127802625	*	&	127814520	&	0.003548	&	1987	&	163343535		&	163352472	&	0.002263	\\									
1847	&	130463281	*	&	130488202	&	0.007305	&	1993	&	164011320	*	&	164013312	&	0.000501	\\									
1861	&	133481040	*	&	133482900	&	0.000537	&	1997	&	164995348		&	165005328	&	0.002502	\\									
1867	&	135499590	*	&	135503322	&	0.00107	&	1999	&	166316517	*	&	166331502	&	0.003749	\\									
1871	&	135629230	*	&	135651670	&	0.00641	&		&			&		&		\\									
1873	&	136086912	*	&	136086912	&	0	&		&			&		&		\\									
1877	&	136953628		&	136963008	&	0.002662	&		&			&		&		\\									
1879	&	138118449		&	138134412	&	0.004521	&		&			&		&		\\									
1889	&	139610048	*	&	139611936	&	0.000529	&		&			&		&		\\									
1901	&	142291000		&	142294800	&	0.001051	&		&			&		&		\\									
1907	&	143639972		&	143649502	&	0.00262	&		&			&		&		\\				

\end{tabular}

\end{sideways}

\end{document}